\documentclass[oneside,english,12pt]{amsart}
\usepackage[T1]{fontenc}
\usepackage[latin9]{inputenc}
\usepackage{amstext}
\usepackage{amsthm}
\usepackage{amssymb}
\PassOptionsToPackage{normalem}{ulem}
\usepackage{ulem}

\makeatletter
\numberwithin{equation}{section}
\numberwithin{figure}{section}
\theoremstyle{plain}
\newtheorem{thm}{\protect\theoremname}
\theoremstyle{definition}
\newtheorem{defn}[thm]{\protect\definitionname}
\theoremstyle{plain}
\newtheorem{prop}[thm]{\protect\propositionname}
\theoremstyle{remark}
\newtheorem{rem}[thm]{\protect\remarkname}
\theoremstyle{plain}
\newtheorem{lem}[thm]{\protect\lemmaname}

\makeatother

\usepackage{babel}
\providecommand{\definitionname}{Definition}
\providecommand{\lemmaname}{Lemma}
\providecommand{\propositionname}{Proposition}
\providecommand{\remarkname}{Remark}
\providecommand{\theoremname}{Theorem}

\begin{document}
\title[Period-two sol. for a class of dist. delay diff. eq.]{Period-two solution for a class of distributed delay differential equations}
\author{Yukihiko Nakata}
\address{Department of Mathematical Sciences, College of Engineering and Science,
Aoyama Gakuin University, 5-10-1 Fuchinobe, Chuo-ku, Sagamihara, Kanagawa
252-5258 Japan}
\begin{abstract}
We study the existence of periodic solutions for differential equations
with distributed delay. It is shown that, for a class of distributed
delay differential equation, a symmetric period-2 solution is obtained
as a periodic solution of a Hamiltonian system of ordinary differential
equations, where the period is twice the maximum delay. This work
extends the result of Kaplan and Yorke (J. Math. Anal. Appl., 1974)
for a discrete delay differential equation with an odd nonlinear function.
To illustrate the result, we present distributed delay differential
equations that have periodic solutions expressed in terms of Jacobi
elliptic functions.

\end{abstract}

\subjclass[2020]{34K13, 37C27} 
\keywords{Distributed delay differential equations; Period-2 solution; Hamiltonian system; Jacobi elliptic functions}

\maketitle

\section{Introduction}

In the study of delay differential equations, it has been shown that
delay is often responsible for causing oscillation of the solution.
Periodic solutions of delay differential equations (DDEs) have been
extensively studied in the literature, see e.g. \cite{KennedyStumpf:2016,Nussbaum:1973,Nussbaum:1974,Nussbaum:19792,Walther:2014}
and references therein. The theory of delay differential equations
has been extensively developed and widely applied to problems in several
disciplines, see \cite{Diekmann:1995,Hale:1993,Smith:2010}.

In the paper \cite{Kaplan=000026Yorke:1974}, Kaplan and Yorke investigate
the existence of a symmetric period-4 solution to the following DDE:
\begin{equation}
x^{\prime}(t)=-f\left(x(t-1)\right),\label{eq:disc_dde}
\end{equation}
where $f$ is an odd continuous function mapping $\mathbb{R}$ to
$\mathbb{R}$. It is assumed that the function $f$ satisfies $xf(x)>0$
for $x\not=0$ with $f(0)=0$. The authors specifically focused on
a periodic solution that satisfies the following symmetry constraint:
\begin{equation}
x(t)=-x(t+2),\label{eq:KY_sym}
\end{equation}
implying the period of the solution is $4$. They showed that the
following planar system of ordinary differential equations (ODEs)
\begin{equation}
x^{\prime}=-f(y),\ y^{\prime}=f(x)\label{eq:KY_ODE}
\end{equation}
yields such a symmetric periodic solution, satisfying (\ref{eq:KY_sym}),
to DDE (\ref{eq:disc_dde}). The proof of Kaplan and Yorke \cite{Kaplan=000026Yorke:1974}
was elaborated by Nussbaum in \cite{Nussbaum:19792}. See also \cite{Walther:2014,KennedyStumpf:2016}.
Stability of the so-called Kaplan-Yorke type periodic solution has
been analyzed in \cite{ChowWalther:1988,Herz:1995}.

The study \cite{Kaplan=000026Yorke:1974} has been extended to a scalar
DDE and a system of DDEs, including (\ref{eq:disc_dde}) as a special
case. In \cite{IvanovDormayer:1998,IvanovDormayer:1999,IvanovLani-Wayda2006},
a scalar DDE incorporating $x(t)$ inside the feedback function is
studied. Fei employs a variational method to analyze the existence
of a symmetric periodic solution for a higher-order Hamiltonian systems
of ODEs to study a scalar DDE with multiple discrete delays \cite{Fei:2006-1,Fei:2006-2}.
Similar results are available for a system of DDEs with multiple discrete
delays \cite{GuoYu:2011,ZhangGuo:2006}. Recently, the author in \cite{Lopez:2020}
characterizes the set of periodic solutions for a scalar DDE, assuming
that feedback function satisfies a certain symmetry condition and
(positive and negative) monotonicity.

In \cite{Nakata:2018} the author of this paper studies the existence
of a period-2 solution for a logistic type differential equation with
a\textit{ distributed} delay: 
\begin{equation}
x'(t)=rx(t)\left(1-\int_{0}^{1}x(t-s)ds\right).\label{eq:logDDE}
\end{equation}
It is shown that a similar ansatz to that in Kaplan and Yorke \cite{Kaplan=000026Yorke:1974}
leads to a system of ODEs that yields a periodic solution to (\ref{eq:logDDE}).
It turns out that equation (\ref{eq:logDDE}) has a periodic solution
expressed in terms of Jacobi elliptic functions. In \cite{Nakata:2021}
the same author also shows that a pendulum equation yields a periodic
solution of a differential equation with a distributed delay. The
study \cite{Nakata:2021} is recently generalized in \cite{XiaoGuo:2023},
where the authors used a variational method. 

Distributed delay appears in modeling of biological processes and
population dynamics, see \cite[Chapter 7]{Smith:2010} and references
therein, transmission of information \cite{Sipahi:2008} and so on.
Linear stability of distributed delay differential equations (distributed
DDEs) has been studied in the literature, see \cite{Bernard:2001,Campbell:2009,Kiss:2010,Miyazaki:1997}
and references therein. In \cite{Walther:1975,Nussbaum:19791,Walther:1977,Kennedy:2014}
a non-ejective fixed point theorems is applied to prove the existence
of a periodic solution for distributed DDEs. We also refer \cite{Azevedoa:2007}
for the study of a coupled system of distributed DDEs. Distributed
delay can be seen as a limit case of multiple discrete delays. Kennedy
studies the existence of periodic solutions for a differential equation
with multiple discrete delays \cite{Kennedy:2009}. In \cite{Kiss:2012},
employing a computational fixed-point method, the authors study the
existence of periodic solutions for DDEs with multiple discrete delays
and show coexistence of multiple periodic solutions. The authors in
\cite{Breda:2016} study the existence of a symmetric periodic solution
for a nonlinear renewal equation, applying the connection between
DDE (\ref{eq:disc_dde}) and the planar system (\ref{eq:KY_ODE}).

In this paper, applying the idea of Kaplan and Yorke \cite{Kaplan=000026Yorke:1974},
we study the existence of a symmetric periodic solution to the following
distributed DDE:
\begin{equation}
x^{\prime}(t)=-g\left(\int_{0}^{1}f(x(t-s))ds\right),\label{eq:DDDE}
\end{equation}
where $f$ and $g$ are continuous functions, satisfying certain additional
properties. If $f$ and $g$ are suitable symmetric odd functions,
it is shown that a period-2 solution, satisfying the symmetric condition
$x(t)=-x(t-1)\ \left(\forall t\in\mathbb{R}\right)$, is given as
a solution of the following system of ODEs:\begin{subequations}\label{eq:Hameq}
\begin{align}
x^{\prime}(t) & =-g(y(t)),\label{eq:Hameq1}\\
y^{\prime}(t) & =2f(x(t)).\label{eq:Hameq2}
\end{align}
\end{subequations}The system of ODEs (\ref{eq:Hameq}) represents
a planar Hamiltonian system with the following Hamiltonian function:
\begin{equation}
\int_{0}^{y}g(\xi)d\xi+2\int_{0}^{x}f(\xi)d\xi.\label{eq:Ham_func}
\end{equation}
Since (\ref{eq:Hameq}) is a Hamiltonian system, for some particular
nonlinear function $f$ and $g$, it is possible to obtain an explicit
form of the periodic solution. The form would help our understanding
to the nature of the periodic solutions to (\ref{eq:DDDE}). In this
paper, we find a symmetric periodic solution of (\ref{eq:DDDE}),
exploitting the planar Hamiltonian system (\ref{eq:Hameq}), and obtain
a sufficient condition for the existence of such a symmetric periodic
solution of (\ref{eq:DDDE}) in Theorem \ref{thm:SPPScond}. We then
recover the results obtained in \cite{Nakata:2018,Nakata:2021} from
the view point of equation (\ref{eq:DDDE}). Kennedy \cite{Kennedy:2014}
studied the existence of a periodic solution for a general distributed
DDE including (\ref{eq:DDDE}) as a special case, employing the fixed
point theorem. We here obtain a slightly different condition from
the one obtained in \cite{Kennedy:2014}, due to our different approach. 

The paper is organized as follows. In Section 2, we characterize a
symmetric period-$2$ solution of (\ref{eq:DDDE}). It is shown that
a period-2 solution yields a family of periodic solutions for a distributed
DDE with a multiplicative parameter. In Section 3, we introduce the
notion of \textit{Special Symmetric Periodic Solution} (SSPS for short),
following \cite{IvanovDormayer:1998,IvanovDormayer:1999}. We show
that an SSPS of (\ref{eq:DDDE}) is related to a periodic solution
of Hamiltonian system of ODEs (\ref{eq:Hameq}), see Theorem \ref{Thm:iff}.
In Section 4, following \cite{Schaaf: 1990}, we obtain an explicit
condition for (\ref{eq:Hameq}) to have a period-2 solution, thus
a condition for (\ref{eq:DDDE}) to have a period-2 solution, that
is an SSPS. In Section 5, we reconsider particular examples of (\ref{eq:DDDE}),
from the previous studies \cite{Nakata:2018,Nakata:2021}. In the
special cases, period-2 solutions are expressed in terms of the Jacobi
elliptic functions. Detailed computations and materials are summarized
in Appendices A, B and C.

\section{Period-$2$ solution}
\begin{defn}
If a solution to (\ref{eq:DDDE}) satisfies $x(t)=x(t+2)$ for any
$t\in\mathbb{R}$, then the solution is called a period-2 solution.
\end{defn}

We impose the following hypotheses on $f$:
\begin{enumerate}
\item[(f1)] $f:\mathbb{R}\to\mathbb{R}$ is a continuous function.
\end{enumerate}
and the following hypotheses on $g$:
\begin{enumerate}
\item[(g1)] $g:\mathbb{R}\to\mathbb{R}$ is a continuous function.
\item[(g2)] $xg(x)>0$ for $x\in\mathbb{R}\setminus\left\{ 0\right\} $ and $g(0)=0$.
\item[(g3)] $g$ is an odd function, i.e., $g(x)=-g(-x)$ for $x\in\mathbb{R}$.
\end{enumerate}
The following Proposition motivates the study of the period-2 solution
to (\ref{eq:DDDE}). 
\begin{prop}
\label{prop:per2}Suppose that (f1), (g1-3) hold. Let $x:\mathbb{R}\to\mathbb{R}$
be a solution of equation (\ref{eq:DDDE}). Then the following statements
are equivalent. 
\begin{enumerate}
\item There exists $t_{0}\in\mathbb{R}$ for any $t\in\mathbb{R}$ such
that $x(t_{0}+t)=x(t_{0}-t)$ holds, i.e., $x$ is an even function
about $t=t_{0}$.
\item There exists $c\in\mathbb{R}$ such that
\begin{equation}
\forall t\in\mathbb{R},\ x(t)+x(t-1)=c\label{eq:x+x(t-1)}
\end{equation}
thus the solution $x$ is a period-2 solution. 
\end{enumerate}
\end{prop}

\begin{proof}
(1) $\Rightarrow$ (2) From the time translation invariance of equation
(\ref{eq:DDDE}), one can assume that $t_{0}=0$ without loss of generality.
Suppose that $x(t)=x(-t),\ \forall t\in\mathbb{R}$. Then, one obtains
\begin{align*}
x^{\prime}(t) & =-x^{\prime}(-t)\\
 & =g\left(\int_{0}^{1}f(x(-t-s))ds\right).\\
 & =g\left(\int_{0}^{1}f(x(t+s))ds\right)\quad(\because x(t)=x(-t))\\
 & =g\left(\int_{0}^{1}f(x(t+1-s))ds\right)\\
 & =-x^{\prime}(t+1),
\end{align*}
which implies that $x(t)+x(t-1)$ is a constant. Therefore, we obtain
the statement (2).

(2) $\Rightarrow$ (1) Let us assume that there exists a constant
$c$ such that for $\forall t\in\mathbb{R}$, $x(t)+x(t-1)=c$. Since
$x(t-1)=c-x(t)$, the solution of (\ref{eq:DDDE}) satisfies the following
system of ODEs \begin{subequations}\label{eq:HP}
\begin{align}
x^{\prime} & =-g(y),\label{eq:H1}\\
y^{\prime} & =f(x)-f(c-x),\label{eq:H2}
\end{align}
\end{subequations}where $y=\int_{0}^{1}f(x(t-s))ds$. Consider the
solution of (\ref{eq:HP}) with the initial condition $(x(t_{0}),y(t_{0}))=(x_{0},0)$
where $x_{0}$ is not a zero of $f(x)-f(c-x)$ and $t_{0}\in\mathbb{R}$
is the initial time. One can easily see that $\left(\tilde{x}(t+t_{0}),\tilde{y}(t+t_{0})\right):=\left(x(t_{0}-t),-y(t_{0}-t)\right)$
solves (\ref{eq:HP}) with the initial condition $\left(\tilde{x}(t_{0}),\tilde{y}(t_{0})\right)=(x(t_{0}),y(t_{0}))=(x_{0},0)$.
From the uniqueness of the solution (see e.g. \cite{Wahlen:2004}),
we obtain $x(t_{0}+t)=x(t_{0}-t)$. We thus obtain the statement (1).
\end{proof}
\begin{rem}
Similar results are obtained in \cite[Lemmas 3.1, 3.2]{Azevedoa:2007}
for a nonlinear renewal equation.
\end{rem}

We have the following result.
\begin{lem}
\label{lem:inv}Let $g(x)=x$. Suppose that (f1) holds. Let $x:\mathbb{R}\to\mathbb{R}$
be a solution of equation (\ref{eq:DDDE}). Assume that there exists
$c\in\mathbb{R}$ such that (\ref{eq:x+x(t-1)}) holds. Then, it holds
that 
\begin{equation}
\forall t\in\mathbb{R},\ \int_{0}^{2}f(x(t-s))ds=0.\label{eq:inv}
\end{equation}
\end{lem}

\begin{proof}
From equation (\ref{eq:DDDE}), one has
\[
x^{\prime}(t)+x^{\prime}(t-1)=-\int_{0}^{2}f(x(t-s))ds.
\]
Since $x(t)+x(t-1)$ is a constant, $x^{\prime}(t)+x^{\prime}(t-1)=0$
for $\forall t\in\mathbb{R}$ follows. Thus (\ref{eq:inv}) holds.
\end{proof}
Then, we show that the period-2 solution of equation (\ref{eq:DDDE})
gives rise to a family of periodic solutions to the equation of the
following form 
\begin{equation}
x^{\prime}(t)=-r\int_{0}^{1}f(x(t-s))ds,\label{eq:req}
\end{equation}
where $r>0$ is a multiplicative parameter. 
\begin{prop}
Let $g(x)=x$. Suppose that (f1) holds. Let $x:\mathbb{R}\to\mathbb{R}$
be a solution of equation (\ref{eq:DDDE}). Assume that there exists
$c\in\mathbb{R}$ such that (\ref{eq:x+x(t-1)}) holds. Then, for
every $n\in\mathbb{N}$, equation (\ref{eq:req}) with $r=\left(2n-1\right)^{2}$
has a periodic solution $x((2n-1)t)$ of period $2/(2n-1)$.
\end{prop}

\begin{proof}
For $n\in\mathbb{N}$, we let $x_{n}(t):=x(\alpha_{n}t)$, where $\alpha_{n}:=2n-1$.
From the direct computation, 
\begin{equation}
x_{n}^{\prime}(t)=\alpha_{n}x^{\prime}(\alpha_{n}t)=-\alpha_{n}\int_{0}^{1}f(x(\alpha_{n}t-s))ds.\label{eq:xn}
\end{equation}
Applying (\ref{eq:inv}) in Lemma \ref{lem:inv}, we see that 
\[
\int_{1}^{\alpha_{n}}f(x(\alpha_{n}t-s))ds=0
\]
holds. Therefore, 
\begin{equation}
\int_{0}^{1}f(x(\alpha_{n}t-s))ds=\int_{0}^{\alpha_{n}}f(x(\alpha_{n}t-s))ds=\alpha_{n}\int_{0}^{1}f(x_{n}(t-s))ds.\label{eq:fx}
\end{equation}
From (\ref{eq:fx}) and (\ref{eq:xn}), we see that $x_{n}(t)$ solves
equation (\ref{eq:req}) with $r=\alpha_{n}^{2}=\left(2n-1\right)^{2}$.
\end{proof}

\section{Distributed DDE (\ref{eq:DDDE}) and Hamiltonian system (\ref{eq:Hameq})}

In this section, we show that distributed DDE (\ref{eq:DDDE}) and
Hamiltonian system of ODEs (\ref{eq:Hameq}) are related via a period-2
solution. 

We now impose the following hypotheses on $f$:
\begin{enumerate}
\item[(f2)] $xf(x)>0$ for $x\in\mathbb{R}\setminus\left\{ 0\right\} $ and $f(0)=0$.
\item[(f3)] $f$ is an odd function, i.e., $f(x)=-f(-x)$ for $x\in\mathbb{R}$.
\end{enumerate}
Then, analogously to \cite{IvanovDormayer:1998,IvanovDormayer:1999},
we define a special symmetric periodic solution (SSPS for short) for
equation (\ref{eq:DDDE}).
\begin{defn}
Let $x:\mathbb{R}\to\mathbb{R}$ be a solution of equation (\ref{eq:DDDE}).
$x$ is called a \textit{special symmetric periodic solution (SSPS)}
of (\ref{eq:DDDE}) if $x$ is a periodic solution of (\ref{eq:DDDE})
with minimal period $2$, satisfying 
\begin{equation}
x(t)+x(t-1)=0,\ \forall t\in\mathbb{R}.\label{eq:SSPS2}
\end{equation}
\end{defn}

Similar to the well-known connection established in Kaplan and Yorke
\cite{Kaplan=000026Yorke:1974} for the symmetric periodic solution
between the discrete DDE (\ref{eq:disc_dde}) and the Hamiltonian
system (\ref{eq:KY_ODE}), we establish the following connection between
the distributed DDE (\ref{eq:DDDE}) and the Hamiltonian system (\ref{eq:Hameq}).
\begin{thm}
\label{Thm:iff}Suppose that (f1-3) and (g1-3) hold. Then the following
statements are true.
\begin{enumerate}
\item Let $x:\mathbb{R}\to\mathbb{R}$ be an SSPS of equation (\ref{eq:DDDE}).
Then 
\begin{equation}
\left(x(t),y(t)\right)=\left(x(t),\int_{0}^{1}f(x(t-s))ds\right)\label{eq:Th5}
\end{equation}
is a periodic solution in $\mathbb{R}^{2}$ with minimal period $2$
for Hamiltonian system (\ref{eq:Hameq}). 
\item Let $\left\{ \left(x(t),y(t)\right):t\in\mathbb{R}\right\} \subset\mathbb{R}^{2}$
be a closed orbit with minimal period $2$ around the origin of Hamiltonian
system (\ref{eq:Hameq}). Then $x:\mathbb{R}\to\mathbb{R}$ is an
SSPS of equation (\ref{eq:DDDE}). 
\end{enumerate}
Thus equation (\ref{eq:DDDE}) has an SSPS if and only if Hamiltonian
system (\ref{eq:Hameq}) has a closed orbit in $\mathbb{R}^{2}$ with
minimal period $2$.
\end{thm}

\begin{proof}
(1) Let $x$ be an SSPS of equation (\ref{eq:DDDE}). Define
\[
y(t):=\int_{0}^{1}f(x(t-s))ds.
\]
Then 
\[
y^{\prime}(t)=f(x(t))-f(x(t-1)).
\]
From (f3) and the property of SSPS that is (\ref{eq:SSPS2}), one
has 
\[
-f(x(t-1))=f(-x(t-1))=f(x(t)).
\]
Therefore, we see that $\left(x(t),y(t)\right)$ satisfies the Hamiltonian
system (\ref{eq:Hameq}). It is obvious that $\left(x(t),y(t)\right)$
is a closed symmetric orbit around the origin of minimal period $2$.

(2) Let $\left\{ \left(x(t),y(t)\right):t\in\mathbb{R}\right\} $
be a closed symmetric trajectory around the origin of minimal period
$2$ for the Hamiltonian system (\ref{eq:Hameq}). Since $\left\{ \left(-x(t),-y(t)\right):t\in\mathbb{R}\right\} $
represents the same trajectory, one can find that 
\[
\left(x(t+1),y(t+1)\right)=\left(-x(t),-y(t)\right),\ t\in\mathbb{R}.
\]
Since $x(t)$ is of period $2$, the property $x(t-1)=-x(t)$ holds.
Therefore, it follows that, using the condition (f3), $f(x(t))=-f(-x(t))=-f(x(t-1))$,
thus
\[
y^{\prime}(t)=2f(x(t))=f(x(t))-f(x(t-1)),
\]
from which we obtain $y(t)=\int_{t-1}^{t}f(x(s))ds+\text{Const}$.
We show that the constant is zero. Consider a closed orbit around
the origin of system (\ref{eq:Hameq}) satisfying $x(\frac{1}{2})=0$
and $y(\frac{1}{2})>0$. From Lemma \ref{lem:sol_Ham} in Appendix
A, $\left(x,y\right)=\left(-x\left(-t\right),y(-t)\right)$ also solve
(\ref{eq:Hameq}). Therefore, from the uniqueness of the solution
of (\ref{eq:Hameq}), we have $\left(x\left(\frac{1}{2}+t\right),y\left(\frac{1}{2}+t\right)\right)=\left(-x\left(\frac{1}{2}-t\right),y\left(\frac{1}{2}-t\right)\right)$.
Thus $x$ is an odd function about $t=\frac{1}{2}$, i.e., $x(t+\frac{1}{2})=-x(t-\frac{1}{2})$.
It then follows that 
\[
y(t)=\int_{t-1}^{t}f(x(s))ds.
\]
Consequently, the first component of the solution of equation (\ref{eq:Hameq})
yields an SSPS $x(t)$ for distributed DDE (\ref{eq:DDDE}).
\end{proof}

\section{Period map of Hamiltonian system (\ref{eq:Hameq})}

In this section, following \cite[Chapters I, II]{Schaaf: 1990}, we
derive an explicit condition for Hamiltonian system (\ref{eq:Hameq})
to have the periodic orbit of period-2 around the origin. We suppose
that (f1-3) and (g1-3) hold. Consider the solution of the Hamiltonian
system (\ref{eq:Hameq}) with the following initial condition 
\begin{equation}
(x(0),y(0))=(x_{0},0),\ x_{0}\in\mathbb{R}.\label{eq:HamIC}
\end{equation}
The trajectory of the solution $(x(t),y(t))$ with the initial condition
(\ref{eq:HamIC}) is a closed orbit in $\mathbb{R}^{2}$, satisfying
\begin{equation}
2F(x)+G(y)=2F(x_{0}),\label{eq:Ham}
\end{equation}
where 
\[
F(x)=\int_{0}^{x}f(\xi)d\xi,\ G(y)=\int_{0}^{y}g(\xi)d\xi.
\]
It is easy to see that the closed orbit given by (\ref{eq:Ham}) with
$x_{0}\not=0$ is symmetric about the $x$-axis and the $y$-axis.

For $x,\ y\in\mathbb{R}$, we implicitly define two variables $u$
and $v$ via\begin{subequations}\label{eq:uv}
\begin{align}
 & 2F(x)=\frac{1}{2}u^{2},\ \text{sign}x=\text{sign}u,\label{eq:u}\\
 & G(y)=\frac{1}{2}v^{2},\ \text{sign}y=\text{sign}v.\label{eq:v}
\end{align}
\end{subequations}Note that $u$ and $v$ are uniquely determined
and that $(u,v)=(0,0)$ for $(x,y)=(0,0)$.

We suppose that 
\begin{enumerate}
\item[(H1)] There exist
\begin{align*}
 & a=\lim_{x\to0}\frac{f(x)}{x},\ A=\lim_{x\to\infty}\frac{f(x)}{x},\\
 & b=\lim_{y\to0}\frac{g(y)}{y},\ B=\lim_{y\to\infty}\frac{g(y)}{y}
\end{align*}
allowing that those quantities are $0$ and $\infty$. 
\item[(H2)] It holds that 
\[
\lim_{x\to\infty}F(x)=\infty,\ \lim_{x\to\infty}G(x)=\infty.
\]
\end{enumerate}
Note that $0\leq a,\ b,\ A,\ B\leq\infty$ and that $u\to\infty$
(as $x\to\infty)$ and $v\to\infty$ (as $y\to\infty)$, from (f1-2)
and (g1-2).

From (\ref{eq:uv}), it is easy to see that, for $u\not=0$, $x$
is differentiable with respect to $u$. Similarly, for $v\not=0$,
$y$ is differentiable with respect to $v$. Denote by $x_{u}$ the
derivative of $x$ with respect to $u$ and $y_{v}$ the derivative
of $y$ with respect to $v$. Then one has 
\begin{equation}
x_{u}=\frac{u}{2f(x)},\ y_{v}=\frac{v}{g(y)}\label{eq:xuyv}
\end{equation}
for $u\not=0$ and $v\not=0$.
\begin{lem}
Suppose that (H1-2) hold. If $a,b\in\left(0,\infty\right)$, then
$x$ and $y$ are differentiable at $u=0$ and $v=0$, respectively,
and \begin{subequations}\label{eq:xuyv0}
\begin{align}
 & x_{u}(0)=\lim_{u\to0}x_{u}(u)=\frac{1}{\sqrt{2a}},\label{eq:xu0}\\
 & y_{v}(0)=\lim_{v\to0}y_{v}(v)=\frac{1}{\sqrt{b}}.\label{eq:yv0}
\end{align}
\end{subequations}If $A,B\in\left(0,\infty\right)$, then it also
holds that\begin{subequations}\label{eq:xuyvinf}
\begin{align}
 & \lim_{u\to\infty}x_{u}(u)=\frac{1}{\sqrt{2A}},\label{eq:xuinf}\\
 & \lim_{v\to\infty}y_{v}(v)=\frac{1}{\sqrt{B}}.\label{eq:yvinf}
\end{align}
\end{subequations}
\end{lem}

\begin{proof}
First, we consider $x_{u}(0)=\lim_{u\to0}x/u$. From (\ref{eq:u}),
one has 
\[
\lim_{u\to0}\frac{x}{u}=\lim_{x\to0}\frac{x}{2\sqrt{F(x)}}.
\]
By l'H\^{o}pital's rule and (\ref{eq:u}), 
\[
\lim_{x\to0}\frac{x^{2}}{4F(x)}=\lim_{x\to0}\frac{2x}{4f(x)}=\frac{1}{2a}.
\]
Thus we obtain $x_{u}(0)=1/\sqrt{2a}$. One also sees that 
\[
\lim_{u\to0}x_{u}(u)=\lim_{u\to0}\frac{u}{2f(x)}
\]
Then by l'H\^{o}pital's rule and (\ref{eq:u}), 
\[
\lim_{u\to0}\frac{u^{2}}{4f^{2}(x)}=\lim_{x\to0}\frac{F(x)}{f(x)^{2}}=\lim_{x\to0}\frac{F(x)}{x^{2}}\frac{1}{f(x)^{2}/x^{2}}=\lim_{x\to0}\frac{f(x)}{2x}\frac{1}{f(x)^{2}/x^{2}}=\frac{1}{2a}.
\]
Hence we obtain (\ref{eq:xu0}). Similarly, we obtain other equalities
in (\ref{eq:xuyv0}) and (\ref{eq:xuyvinf}).
\end{proof}
From the proof in Lemma 8, we also have the following equalities:
\begin{align*}
\lim_{u\to0}x_{u}(u) & =\begin{cases}
\infty & \text{if }a=0,\\
0 & \text{if }a=\infty,
\end{cases} & \lim_{u\to\infty}x_{u}(u) & =\begin{cases}
\infty & \text{if }A=0,\\
0 & \text{if }A=\infty,
\end{cases}\\
\lim_{v\to0}y_{v}(v) & =\begin{cases}
\infty & \text{if }b=0,\\
0 & \text{if }b=\infty,
\end{cases} & \lim_{v\to\infty}y_{v}(v) & =\begin{cases}
\infty & \text{if }B=0,\\
0 & \text{if }B=\infty.
\end{cases}
\end{align*}

Since $x^{\prime}=x_{u}u^{\prime}$ and $y^{\prime}=y_{v}v^{\prime}$,
where $u^{\prime}$ and $v^{\prime}$ are derivative of $u$ and $v$
with respect to $t$, from (\ref{eq:Hameq}) and (\ref{eq:xuyv}),
one obtains 
\begin{equation}
x_{u}y_{v}u^{\prime}=-v,\ x_{u}y_{v}v^{\prime}=u.\label{eq:uveq}
\end{equation}
Consider the time transformation defined by 
\[
\frac{dt}{d\theta}=x_{u}y_{v},\ t(0)=0,
\]
then, from (\ref{eq:uveq}), we obtain the following equations:
\begin{align*}
 & \frac{du}{d\theta}=-v,\ \frac{dv}{d\theta}=u.
\end{align*}
Hence we have 
\[
(u,v)=(p\cos\theta,p\sin\theta),
\]
where $p=2\sqrt{F(x_{0})}$ from (\ref{eq:u}). From (f3) and (g3),
the period of the solution $T=T(p)$ is given as 
\[
T(p)=4\int_{0}^{\frac{\pi}{2}}x_{u}(p\cos\theta)y_{v}(p\sin\theta)d\theta.
\]

Finally, applying the Lebesgue\textquoteright s Dominated Convergence
Theorem, as in \cite[Proposition 1.6.1]{Schaaf: 1990} and \cite[Section 5]{Villadelprat1:2020},
we obtain the asymptotic behaviour of the period map as follows.
\begin{lem}
Suppose that (H1-2) hold. One has 
\[
\lim_{p\to+0}T(p)=\sqrt{\frac{2}{ab}}\pi,\ \lim_{p\to\infty}T(p)=\sqrt{\frac{2}{AB}}\pi,
\]
with the convention $\frac{1}{0}=\infty$ and $\frac{1}{\infty}=0$.
\end{lem}

Therefore, we obtain the existence of the period-2 solution to (\ref{eq:DDDE}),
according to Theorem \ref{Thm:iff}.
\begin{thm}
\label{thm:SPPScond}Suppose that (f1-3), (g1-3) and (H1-2) hold.
If $a,b,A$ and $B$ satisfy either
\[
ab<\frac{\pi^{2}}{2}<AB,\text{ or }AB<\frac{\pi^{2}}{2}<ab,
\]
then system (\ref{eq:Hameq}) has a closed orbit with minimal period
$2$, thus equation (\ref{eq:DDDE}) has an SSPS.
\end{thm}

\section{Examples}

In this section we study distributed DDE (\ref{eq:DDDE}) with particular
nonlinear functions. We express the periodic solutions in terms of
the Jacobi elliptic functions. 

Let us introduce the complete elliptic integrals of the first kind
and of the second kind respectively given as 
\begin{align*}
K(k) & =\int_{0}^{\frac{\pi}{2}}\frac{1}{\sqrt{1-k^{2}\sin^{2}\theta}}d\theta,\\
E(k) & =\int_{0}^{\frac{\pi}{2}}\sqrt{1-k^{2}\sin^{2}\theta}d\theta
\end{align*}
for $0\leq k<1$. We refer \cite{Byrd,Meyer:2001} for the fundamental
properties of the Jacobi elliptic functions. 

\subsection{$f(x)=r\sin(x)$ and $g(x)=x$}

In \cite{Nakata:2021}, we consider the existence of an SSPS to the
following equation: 
\begin{equation}
x^{\prime}(t)=-g\left(\int_{0}^{1}x(t-s)ds\right),\label{eq:Ex1}
\end{equation}
with $g(x)=r\sin x,$ where $r>0$. In Appendix B, we show a transformation
between equations (\ref{eq:Ex1}) and (\ref{eq:DDDE}) with $g(x)=x$.
Here, we consider the existence of an SSPS of (\ref{eq:DDDE}) with
$f(x)=r\sin x$ and $g(x)=x$. One can see that the system of ODEs
(\ref{eq:Hameq}) with $f(x)=r\sin x$ and $g(x)=x$ is a well-known
pendulum equation. We denote by $\text{cn},\ \text{sn},\ \text{dn}$
Jacobi elliptic functions. See \cite{Byrd,Meyer:2001} for the fundamental
properties of the Jacobi elliptic functions. We then obtain the following
result.
\begin{thm}
\label{thm:Ex1}Let $f(x)=r\sin x,\ r>0$ and $g(x)=x$. Equation
(\ref{eq:DDDE}) has an SSPS if and only if $r>\pi^{2}/2$. The SSPS
is expressed as 
\begin{equation}
x(t)=2\arcsin\left(k\:\text{sn}\left(\sqrt{2r}\:t,k\right)\right),\label{eq:SSPS_sin}
\end{equation}
where $k\in\left(0,1\right)$ is uniquely determined as the solution
of the following equation: 
\begin{equation}
\frac{2K\left(k\right)}{\sqrt{2r}}=1.\label{eq:PerSSPS}
\end{equation}
\end{thm}

\begin{proof}
We seek an SSPS $x(t)$ satisfying $-\pi<x(t)<\pi$ for $t\in\mathbb{R}$
for (\ref{eq:DDDE}), applying Theorem \ref{Thm:iff}. In Appendix
C, we compute a periodic orbit for (\ref{eq:Hameq}), surrounding
the origin. (\ref{eq:Hameq}) has a periodic solution given as in
(\ref{eq:pensol}) in Lemma \ref{lem:pend}, where the period is $T=T(k)$
given as in (\ref{eq:Per_Ex1}). Since $K(k)$ is an increasing function
of $k$ with $K(0)=\pi/2$, there exists a unique $k\in(0,1)$ such
that $T(k)=2$ if and only if $r>\pi^{2}/2$. Thus, for $r>\pi^{2}/2$,
one can see that there exists a unique $k\in(0,1)$ such that (\ref{eq:PerSSPS})
holds and that (\ref{eq:SSPS_sin}) is the SSPS for (\ref{eq:DDDE}).
\end{proof}

\subsection{$f(x)=r\left(e^{x}-1\right)$ and $g(x)=x$}

In \cite{Nakata:2018} we study logistic equation with distributed
delay (\ref{eq:logDDE}). Changing the variable $x\mapsto\log x$,
we obtain equation (\ref{eq:DDDE}) with $f(x)=r(e^{x}-1)$ and $g(x)=x$.
Note that $f(x)$ is \uline{not} an odd function, thus we may not
directly apply Theorem \ref{Thm:iff} to obtain the period-2 solution
of (\ref{eq:DDDE}).

Consider the existence of a period-2 solution, satisfying 
\begin{equation}
x(t)+x(t-1)=c,\label{eq:c_symm}
\end{equation}
where $c$ denotes a constant. The period-2 solution satisfies the
following system of ODEs:\begin{subequations}\label{eq:Ham2}
\begin{align}
x^{\prime} & =-y,\label{eq:Ham21}\\
y^{\prime} & =f(x)-f(c-x),\label{eq:Ham22}
\end{align}
\end{subequations}(see also the proof of Proposition \ref{prop:per2}).
Note that $c$ is a parameter to be determined. 

Let $w=x-\frac{1}{2}c$. Then system (\ref{eq:Ham2}) becomes 
\begin{align*}
w^{\prime} & =-y,\\
y^{\prime} & =f\left(w+\frac{1}{2}c\right)-f\left(\frac{1}{2}c-w\right).
\end{align*}
For $f(x)=r(e^{x}-1)$, we see 
\begin{align*}
f\left(w+\frac{1}{2}c\right)-f\left(\frac{1}{2}c-w\right) & =r\left(e^{w+\frac{1}{2}c}-1\right)-r\left(e^{-w+\frac{1}{2}c}-1\right)\\
 & =2re^{\frac{1}{2}c}\sinh w.
\end{align*}
Thus we obtain the following system of ODEs:\begin{subequations}\label{eq:wy}
\begin{align}
w^{\prime} & =-y,\label{eq:wy1}\\
y^{\prime} & =2re^{\frac{1}{2}c}\sinh w.\label{eq:wy2}
\end{align}
\end{subequations}

Let us first find a period-2 solution to system (\ref{eq:wy}). Consider
the periodic solution of (\ref{eq:wy}) with the following initial
condition 
\begin{equation}
\left(w(0),y(0)\right)=\left(w_{0},0\right),\ w_{0}>0.\label{eq:wyic}
\end{equation}
Let 
\[
\gamma=re^{\frac{1}{2}c},\quad\alpha=\sinh\left(\frac{w_{0}}{2}\right)>0.
\]
Define 
\begin{equation}
\beta=\sqrt{2\gamma(1+\alpha^{2})},\quad k=\frac{\alpha}{\sqrt{1+\alpha^{2}}}\in\left(0,1\right).\label{eq:betaK}
\end{equation}
Note that 
\begin{equation}
\beta k=\alpha\sqrt{2\gamma},\quad\alpha=\frac{k}{\sqrt{1-k^{2}}}\label{eq:betak2}
\end{equation}
hold. In terms of $\beta$ and $k$, the solution of (\ref{eq:wy})
is expressed as follows, see Appendix C for the proof.
\begin{lem}
\label{lem:Sol_wy}The solution of system (\ref{eq:wy}) with the
initial condition (\ref{eq:wyic}) is expressed as\begin{subequations}\label{eq:wysol}
\begin{align}
 & w(t)=\log\left(\frac{\text{dn}\left(\beta t,k\right)+k\text{cn}\left(\beta t,k\right)}{\text{dn}\left(\beta t,k\right)-k\text{cn}\left(\beta t,k\right)}\right),\label{eq:wysol1}\\
 & y(t)=2\beta k\text{sn}\left(\beta t,k\right).\label{eq:wysol2}
\end{align}
\end{subequations}
\end{lem}

The period of the solution (\ref{eq:wysol}) for (\ref{eq:wy}) is
2 if and only if 
\begin{align}
 & \beta=2K(k),\label{eq:betak_cond}
\end{align}
where $K$ denotes the complete elliptic integrals of the first kind.

Furthermore, $x=w+\frac{1}{2}c$ satisfies (\ref{eq:Th5}) with $y(0)=0$
if and only if 
\[
0=-\int_{0}^{1}f(x(-s))ds.
\]
Finally, we obtain the following result.
\begin{thm}
\label{thm:Ex2}Let $f(x)=r(e^{x}-1),\ r>0$ and $g(x)=x$. Equation
(\ref{eq:DDDE}) has a period-2 solution, satisfying (\ref{eq:c_symm})
for some $c\in\mathbb{R}$, if and only if $r>\pi^{2}/2$. The period-2
solution is expressed as 
\begin{equation}
x(t)=\log\left(\frac{K(k)\left(\text{dn}\left(2K(k)t,k\right)+k\text{cn}\left(2K(k)t,k\right)\right)^{2}}{2E(k)-K(k)(1-k^{2})}\right),\label{eq:x(t)_log}
\end{equation}
where $k\in\left(0,1\right)$ is a unique solution of 
\begin{equation}
r=2K(k)\left(2E(k)-K(k)(1-k^{2})\right)\label{eq:req_Ex2}
\end{equation}
for $r>\pi^{2}/2$.
\end{thm}

\begin{proof}
From Lemma \ref{lem:Sol_wy}, we can express $x(t)$ as
\begin{align}
 & x(t)=\log\left(\frac{\text{dn}\left(\beta t,k\right)+k\text{cn}\left(\beta t,k\right)}{\text{dn}\left(\beta t,k\right)-k\text{cn}\left(\beta t,k\right)}\right)+\frac{1}{2}c.\label{eq:x_sol}
\end{align}
We determine two parameters $k$ and $c$ so that (\ref{eq:x_sol})
is a period-$2$ solution of equation (\ref{eq:DDDE}). 

Equation (\ref{eq:betak_cond}) holds if and only if the period of
(\ref{eq:wysol1}) and (\ref{eq:wysol2}) is 2. We thus assume that
(\ref{eq:betak_cond}) holds. Then 
\begin{equation}
re^{\frac{1}{2}c}=2K^{2}(k)(1-k^{2}).\label{eq:rc_rel}
\end{equation}
Furthermore we require that 
\[
y(t)=\int_{0}^{1}f(x(t-s))ds
\]
holds for $t\in\mathbb{R}$, where $y(t)$ is (\ref{eq:wysol2}).

Since, from (\ref{eq:betak_cond}), $y(1)=2\beta k\text{sn}\left(\beta,k\right)=0$
follows, we must have
\begin{equation}
0=\int_{0}^{1}\left(e^{x(t)}-1\right)dt.\label{eq:condpr}
\end{equation}
Substituting (\ref{eq:x_sol}) into (\ref{eq:condpr}), we obtain
\begin{align*}
0 & =e^{\frac{1}{2}c}\int_{0}^{1}\frac{\text{dn}\left(\beta t,k\right)+k\text{cn}\left(\beta t,k\right)}{\text{dn}\left(\beta t,k\right)-k\text{cn}\left(\beta t,k\right)}dt-1.
\end{align*}
We now compute 
\begin{align*}
\frac{\text{dn}\left(\beta t,k\right)+k\text{cn}\left(\beta t,k\right)}{\text{dn}\left(\beta t,k\right)-k\text{cn}\left(\beta t,k\right)} & =\frac{\left(\text{dn}\left(\beta t,k\right)+k\text{cn}\left(\beta t,k\right)\right)^{2}}{1-k^{2}}\\
 & =\frac{2\text{dn}^{2}\left(\beta t,k\right)+2k\text{dn}\left(\beta t,k\right)\text{cn}\left(\beta t,k\right)}{1-k^{2}}-1.
\end{align*}
Then, we obtain
\[
\int_{0}^{1}\frac{\text{dn}\left(\beta t,k\right)+k\text{cn}\left(\beta t,k\right)}{\text{dn}\left(\beta t,k\right)-k\text{cn}\left(\beta t,k\right)}dt=\frac{2}{1-k^{2}}\frac{E(k)}{K(k)}-1,
\]
thus 
\begin{equation}
e^{\frac{1}{2}c}\left(\frac{2}{1-k^{2}}\frac{E(k)}{K(k)}-1\right)-1=0.\label{eq:EK}
\end{equation}
Here, we use 
\[
\int_{0}^{1}\text{dn}^{2}\left(\beta t,k\right)dt=\frac{2}{\beta}E(k)=\frac{E(k)}{K(k)},
\]
(See also 314.02 in \cite{Byrd} for the former identity) and 
\[
\int_{0}^{1}\text{dn}\left(\beta t,k\right)\text{cn}\left(\beta t,k\right)dt=\frac{1}{\beta}\left[\text{sn}\left(\beta t,k\right)\right]_{t=0}^{t=1}=0,
\]
from the condition (\ref{eq:betak_cond}). 

From (\ref{eq:rc_rel}) and (\ref{eq:EK}), we obtain equation (\ref{eq:req_Ex2}).
From Lemma 6 in \cite{Nakata:2018}, one can see that (\ref{eq:req_Ex2})
has a unique solution for $r>\pi^{2}/2$. Then, from (\ref{eq:rc_rel}),
we obtain that 
\[
\frac{1}{2}c=\log\frac{2K^{2}(k)(1-k^{2})}{r}=\log\frac{K(k)(1-k^{2})}{2E(k)-K(k)(1-k^{2})}.
\]
Finally, from (\ref{eq:x_sol}), it follows that 
\begin{align*}
x(t) & =\log\left(\frac{\text{dn}\left(\beta t,k\right)+k\text{cn}\left(\beta t,k\right)}{\text{dn}\left(\beta t,k\right)-k\text{cn}\left(\beta t,k\right)}\frac{K(k)(1-k^{2})}{2E(k)-K(k)(1-k^{2})}\right)\\
 & =\log\left(\frac{K(k)\left(\text{dn}\left(\beta t,k\right)+k\text{cn}\left(\beta t,k\right)\right)^{2}}{2E(k)-K(k)(1-k^{2})}\right).
\end{align*}
Using (\ref{eq:betak_cond}), we obtain the expression (\ref{eq:x(t)_log}).
\end{proof}

\subsection*{Acknowledgement}

The author would like to thank the editor and the anonymous referee who provided
thoughtful comments which improved the original manuscript. The author
is grateful to Prof. Satoshi Tanaka (Tohoku University) for introducing
the book \cite{Schaaf: 1990} to the author. The author was supported
by JSPS Kakenhi Grant Number 20K03734 and 19K21836.

\appendix

\section{Symmetry of the solution of Hamiltonian system (\ref{eq:Hameq})
with odd nonlinearity}

In this section we suppose that (f1-3) and (g1-3) hold. Consider the
Hamiltonian system of ODEs (\ref{eq:Hameq}). 
\begin{lem}
\label{lem:sol_Ham}Suppose that (f1-3) and (g1-3) hold. Let $(x(t),y(t))$
be a solution of (\ref{eq:Hameq}). Then $\left(-x(t),-y(t)\right),\ \left(x(-t),-y(-t)\right),\ \left(-x\left(-t\right),y(-t)\right)$
are also solutions of (\ref{eq:Hameq}). 
\end{lem}

By direct substitution, one can prove Lemma \ref{lem:sol_Ham} using
the symmetric condition on $f$ and $g$. We here omit the detail.
\begin{prop}
Suppose that (f1-3) and (g1-3) hold. Let $(x(t),y(t))$ be a periodic
solution of (\ref{eq:Hameq}) of period $T$ with the initial condition
(\ref{eq:HamIC}). Then the following statements are true.
\begin{enumerate}
\item $x(t)$ is an even function of $t$ and $y(t)$ is an odd function
of $t$. 
\item $x\left(\frac{1}{4}T+t\right)$ is an odd function of $t$ and $y\left(\frac{1}{4}T+t\right)$
is an even function of $t$. 
\end{enumerate}
In particular, one has 
\begin{align*}
\left(x\left(t\right),y(t)\right) & =\left(-x\left(\frac{1}{2}T-t\right),y\left(\frac{1}{2}T-t\right)\right)\\
 & =\left(-x\left(\frac{1}{2}T+t\right),-y\left(\frac{1}{2}T+t\right)\right)\\
 & =\left(x\left(T-t\right),-y\left(T-t\right)\right).
\end{align*}

\end{prop}

\section{Equivalent formulation}

Here we show that equation (\ref{eq:Ex1}) can be transformed to (\ref{eq:DDDE})
with $g(x)=x$. Let $g(x)=x$. Denote by $x$ a solution of (\ref{eq:DDDE}).
For a solution of (\ref{eq:DDDE}), we define a function $v:\left[-1,\infty\right)\to\mathbb{R}$
by 
\begin{equation}
v^{\prime}=-f(x)\label{eq:veq}
\end{equation}
 with $x(0)=\int_{0}^{1}v(-s)ds.$ Integrating (\ref{eq:veq}), it
follows that 
\[
v(t)-v(t-1)=-\int_{t-1}^{t}f(x(s))dt.
\]
From (\ref{eq:DDDE}), one sees that $v(t)-v(t-1)=x^{\prime}(t)$.
Therefore, we find the following equality: 
\begin{equation}
x(t)=\int_{0}^{1}v(t-s)ds.\label{eq:xv}
\end{equation}
Substituting (\ref{eq:xv}) into equation (\ref{eq:veq}), we obtain
the following equation:
\[
v^{\prime}(t)=-f\left(\int_{0}^{1}v(t-s)ds\right).
\]

\section{Explicit periodic solutions}

In this section, we derive the expressions of periodic solutions of
Hamiltonian system (\ref{eq:Hameq}) for $f(x)=r\sin x,\ r>0$ and
$g(x)=x$ in Section C.1 and for $f(x)=r\left(e^{x}-1\right),\ r>0$
and $g(x)=x$ in Section C.2. 

\subsection{$f(x)=r\sin x$ and $g(x)=x$}

Let $f(x)=r\sin x,\ r>0$ and $g(x)=x$. We consider the periodic
solution of (\ref{eq:Hameq}) with the following initial condition:
\begin{equation}
\left(x(0),y(0)\right)=\left(x_{0},0\right),\ x_{0}\in\left(0,\pi\right).\label{eq:IC1}
\end{equation}
From (\ref{eq:Ham}) we have 
\begin{equation}
y(t)^{2}+8r\sin^{2}\frac{x(t)}{2}=8r\sin^{2}\frac{x_{0}}{2}.\label{eq:Ham1}
\end{equation}
Let $u=\sin\frac{x}{2}$. From (\ref{eq:Hameq1}) and (\ref{eq:Ham1}),
we obtain 
\begin{align}
\left(u^{\prime}\right)^{2} & =2r\left(1-u^{2}\right)\left(k^{2}-u^{2}\right),\label{eq:ueq}
\end{align}
where 
\begin{equation}
k=\sin\frac{x_{0}}{2}\in\left(0,1\right).\label{eq:k1}
\end{equation}

\begin{lem}
\label{lem:pend}Let $f(x)=r\sin x,\ r>0$ and $g(x)=x$. The periodic
solution of (\ref{eq:Hameq}) with the initial condition (\ref{eq:IC1})
is expressed as\begin{subequations}\label{eq:pensol} \textup{
\begin{align}
x(t) & =2\arcsin\left(k\text{sn}\left(\sqrt{2r}\:t+K(k),k\right)\right),\label{eq:Sol1}\\
y(t) & =-2\sqrt{2r}k\text{cn}\left(\sqrt{2r}\:t+K(k),k\right).\label{eq:Sol2}
\end{align}
}\end{subequations}where $k$ is defined in (\ref{eq:k1}). The period
of the solution is given as $T$, where 
\begin{equation}
T=\frac{4K(k)}{\sqrt{2r}}.\label{eq:Per_Ex1}
\end{equation}
\end{lem}

\begin{proof}
Equation (\ref{eq:ueq}) has the solution of the form: 
\[
u(t)=k\ \text{sn}\left(\sqrt{2r}\:t+c,k\right),
\]
where $c$ is determined by the initial condition $u(0)=\sin\frac{x_{0}}{2}=k$,
so that $\text{sn}\left(c,k\right)=1$. Then we see $c=K(k)$ for
$0\leq c\leq4K(k)$. From the definition, we have $x=2\arcsin u$.
Thus we obtain (\ref{eq:Sol1}). From (\ref{eq:Hameq1}), one has
$y(t)=-x^{\prime}(t)$. Thus we obtain (\ref{eq:Sol2}).
\end{proof}

\subsection{$f(x)=r\left(e^{x}-1\right)$ and $g(x)=x$}

Let $f(x)=r(e^{x}-1),\ r>0$ and $g(x)=x$. We consider a periodic
solution of (\ref{eq:wy}) with the initial condition (\ref{eq:wyic}).
From (\ref{eq:Ham}) we have 
\begin{equation}
y^{2}+8\gamma\sinh^{2}\left(\frac{w}{2}\right)=8\gamma\sinh^{2}\left(\frac{w_{0}}{2}\right),\label{eq:Ham_Ex2}
\end{equation}
using $\cosh w=1+2\sinh^{2}\frac{w}{2}$. Let 
\begin{equation}
u:=\sinh\left(\frac{w}{2}\right).\label{eq:Defu_Ex2}
\end{equation}
 Then, from (\ref{eq:wy}) and (\ref{eq:Ham_Ex2}), we obtain 
\begin{equation}
\left(u^{\prime}\right)^{2}=\frac{1}{4}\cosh^{2}\left(\frac{w}{2}\right)y^{2}=2\gamma\left(1+u^{2}\right)\left(\alpha^{2}-u^{2}\right).\label{eq:ueq2}
\end{equation}

\begin{lem}
\label{lem:Ex2}The periodic solution of (\ref{eq:ueq2}) with the
initial condition $u(0)=\alpha$ is expressed as
\begin{equation}
u(t)=\alpha\text{cn}\left(\beta t,k\right),\label{eq:ut_Ex2}
\end{equation}
where the period is given as 
\begin{equation}
T=\frac{4K\left(k\right)}{\beta}.\label{eq:Per_Ex2}
\end{equation}
 Here $\beta$ and $k$ are defined in (\ref{eq:betaK}).
\end{lem}

\begin{proof}
We consider the solution of the following form: 
\[
u(t)=\sqrt{c_{1}}\text{cn}\left(\sqrt{c_{2}}t,\sqrt{c_{3}}\right).
\]
By a direct computation, one sees
\begin{align*}
\left(u^{\prime}\right)^{2} & =c_{1}c_{2}\text{sn}^{2}\left(\sqrt{c_{2}}t,\sqrt{c_{3}}\right)\text{dn}^{2}\left(\sqrt{c_{2}}t,\sqrt{c_{3}}\right)\\
 & =c_{1}c_{2}\left(1-\text{cn}^{2}\left(\sqrt{c_{2}}t,\sqrt{c_{3}}\right)\right)\left(1-c_{3}\text{sn}^{2}\left(\sqrt{c_{2}}t,\sqrt{c_{3}}\right)\right)\\
 & =c_{1}c_{2}\left(1-\text{cn}^{2}\left(\sqrt{c_{2}}t,\sqrt{c_{3}}\right)\right)\left(1-c_{3}\left(1-\text{cn}^{2}\left(\sqrt{c_{2}}t,\sqrt{c_{3}}\right)\right)\right)\\
 & =c_{1}c_{2}\left(1-\frac{1}{c_{1}}u^{2}\right)\left(1-c_{3}+\frac{c_{3}}{c_{1}}u^{2}\right).
\end{align*}
From (\ref{eq:ueq2}), we get \begin{subequations}\label{eq:coe}
\begin{align}
 & c_{1}c_{2}\left(1-c_{3}\right)=2\gamma\alpha^{2},\label{eq:c1}\\
 & c_{2}\left(c_{3}-\left(1-c_{3}\right)\right)=2\gamma\left(\alpha^{2}-1\right),\label{eq:c2}\\
 & \frac{c_{2}c_{3}}{c_{1}}=2\gamma.\label{eq:c3}
\end{align}
\end{subequations}From (\ref{eq:coe}), eliminating $c_{1}$ and
$c_{2}$, we have
\[
\frac{\alpha^{2}}{\left(\alpha^{2}-1\right)^{2}}=\frac{c_{3}(1-c_{3})}{\left(2c_{3}-1\right)^{2}},
\]
from which we obtain
\[
0=c_{3}^{2}-c_{3}+\frac{\alpha^{2}}{\left(\alpha^{2}+1\right)^{2}}=\left(c_{3}-\frac{\alpha^{2}}{1+\alpha^{2}}\right)\left(c_{3}-\frac{1}{1+\alpha^{2}}\right).
\]
Hence 
\[
c_{3}=\frac{\alpha^{2}}{1+\alpha^{2}},\ \frac{1}{1+\alpha^{2}}\in\left[0,1\right].
\]
Then, from (\ref{eq:c2}), we have
\[
c_{2}=\begin{cases}
2\gamma(1+\alpha^{2}) & \text{if }c_{3}=\frac{\alpha^{2}}{1+\alpha^{2}},\\
-2\gamma(1+\alpha^{2}) & \text{if }c_{3}=\frac{1}{1+\alpha^{2}}.
\end{cases}
\]
From positivity of $c_{2}$, one obtains $c_{3}=\frac{\alpha^{2}}{1+\alpha^{2}}$.
Therefore, we now have
\begin{align*}
\left(c_{1},c_{2},c_{3}\right) & =\left(\alpha^{2},2\gamma(1+\alpha^{2}),\frac{\alpha^{2}}{1+\alpha^{2}}\right)=\left(\alpha^{2},\beta^{2},k^{2}\right).
\end{align*}
Hence we obtain (\ref{eq:Sol1}). The period of the solution is given
as in (\ref{eq:Per_Ex2}), from the property of the Jacobi elliptic
function $\text{cn}$. 
\end{proof}

\begin{proof}[Proof of Lemma \ref{lem:Sol_wy}]
From (\ref{eq:ut_Ex2}) in Lemma \ref{lem:Ex2}, we have 
\begin{align*}
w(t) & =2\sinh^{-1}\left(\alpha\text{cn}\left(\beta t,k\right)\right)\\
 & =2\log\left(\alpha\text{cn}\left(\beta t,k\right)+\sqrt{1+\alpha^{2}\text{cn}^{2}\left(\beta t,k\right)}\right),
\end{align*}
where $\beta$ and $k$ are defined in (\ref{eq:betaK}). Then, using
(\ref{eq:betak2}), we obtain
\begin{align*}
w(t) & =2\log\left(\frac{k\text{cn}\left(\beta t,k\right)+\sqrt{1-k^{2}+k^{2}\text{cn}^{2}\left(\beta t,k\right)}}{\sqrt{1-k^{2}}}\right)\\
 & =2\log\left(\frac{k\text{cn}\left(\beta t,k\right)+\text{dn}\left(\beta t,k\right)}{\sqrt{1-k^{2}}}\right)\\
 & =\log\left(\frac{\text{dn}\left(\beta t,k\right)+k\text{cn}\left(\beta t,k\right)}{\text{dn}\left(\beta t,k\right)-k\text{cn}\left(\beta t,k\right)}\right).
\end{align*}
Note that $\text{dn}^{2}\left(\beta t,k\right)-k^{2}\text{cn}^{2}\left(\beta t,k\right)=1-k^{2}$
holds. Hence we obtain (\ref{eq:wysol1}). From (\ref{eq:Ham_Ex2})
\begin{align*}
y(t)^{2} & =8\gamma\left(\alpha^{2}-\sinh^{2}\left(\frac{w}{2}\right)\right)\\
 & =8\gamma\alpha^{2}\left(1-\text{cn}^{2}\left(\beta t,k\right)\right)\\
 & =8\gamma\alpha^{2}\text{sn}^{2}\left(\beta t,k\right).
\end{align*}
Hence we obtain (\ref{eq:wysol2}).
\end{proof}

\end{document}